\documentclass[12pt]{amsart}
\usepackage{amssymb}
\usepackage{amsmath}
\numberwithin{equation}{section}
\usepackage{breakcites}
\usepackage{color}
\usepackage{graphicx}
\usepackage{epstopdf}

\newtheorem{thm}{Theorem}[section]

\newtheorem{cor}[thm]{Corollary}
\newtheorem{prop}[thm]{Proposition}
\newtheorem{con}[thm]{Conjecture}

\theoremstyle{definition}
\newtheorem{definition}[thm]{Definition}

\newtheorem{defn}[thm]{Definition}

\theoremstyle{remark}
\newtheorem{remark}[thm]{Remark}

\renewcommand{\S}{\mathfrak S}
\newcommand{\s}{\sigma}

\newcommand\C{{\mathbb{C}}}
\newcommand\R{{\mathbb R}}

\newcommand\Z{{\mathbb{Z}}}

\newcommand\PP{{\mathbb{P}}}
\newcommand\N{{\mathbb{N}}}

\newcommand\ps{{\rm ps}}

\newcommand\bq{\begin{equation}}
\newcommand\eq{\end{equation}}
\newcommand\beq{\begin{eqnarray*}}
\newcommand\eeq{\end{eqnarray*}}
\newcommand\ben{\begin{enumerate}}
\newcommand\een{\end{enumerate}}
\newcommand\bit{\begin{itemize}}
\newcommand\eit{\end{itemize}}

\newcommand\des{{\rm des}}
\newcommand\exc{{\rm exc}}
\newcommand\inv{{\rm inv}}
\newcommand\maj{{\rm maj}}

\newcommand\sg{{\mathfrak S}}
\newcommand\Des{{\rm DES}}

\newcommand\fix{{\rm fix}}
\newcommand\ch{{\rm ch}}

\newcommand\x{{\mathbf x}}

\def\zz{{\mathbb Z}}

\def\ff{{\mathbb F}}

\def\pp{{\mathbb P}}

\def\rr{{\mathbb R}}

\def\ff{{\mathcal F}}

\newlength\cellsize 
\savebox2{%
\begin{picture}(15,15)
\put(0,0){\line(1,0){15}}
\put(0,0){\line(0,1){15}}
\put(15,0){\line(0,1){15}}
\put(0,15){\line(1,0){15}}
\end{picture}}
\newcommand\cellify[1]{\def\thearg{#1}\def\nothing{}%
\ifx\thearg\nothing
\vrule width0pt height\cellsize depth0pt\else
\hbox to 0pt{\usebox2\hss}\fi%
\vbox to 15\unitlength{
\vss
\hbox to 15\unitlength{\hss$#1$\hss}
\vss}}
\newcommand\tableau[1]{\vtop{\let\\=\cr
\setlength\baselineskip{-16000pt}
\setlength\lineskiplimit{16000pt}
\setlength\lineskip{0pt}
\halign{&\cellify{##}\cr#1\crcr}}}
\savebox3{%
\begin{picture}(15,15)
\put(0,0){\line(1,0){15}}
\put(0,0){\line(0,1){15}}
\put(15,0){\line(0,1){15}}
\put(0,15){\line(1,0){15}}
\end{picture}}
\newcommand\expath[1]{%
\hbox to 0pt{\usebox3\hss}%
\vbox to 15\unitlength{
\vss
\hbox to 15\unitlength{\hss$#1$\hss}
\vss}}

\begin{document}

\title[Variations of Eulerian polynomials]{Gamma-positivity of variations of Eulerian polynomials}

\author[Shareshian]{John Shareshian$^1$}
\address{Department of Mathematics, Washington University, St. Louis, MO 63130}
\thanks{$^{1}$Supported in part by NSF Grants
DMS 1202337 and DMS 1518389}
\email{shareshi@math.wustl.edu}

\author[Wachs]{Michelle L. Wachs$^2$}
\address{Department of Mathematics, University of Miami, Coral Gables, FL 33124}
\email{wachs@math.miami.edu}
\thanks{$^{2}$Supported in part by NSF Grants
DMS 1202755 and DMS 1502606}

\begin{abstract} An identity of Chung, Graham and Knuth involving binomial coefficients and Eulerian numbers 
motivates our study of a class of polynomials that we call  binomial-Eulerian polynomials.  These polynomials share several properties with the Eulerian polynomials.  For one thing, they are   $h$-polynomials of  simplicial polytopes, which gives a geometric interpretation of the fact that they are palindromic and unimodal.   A formula of Foata and Sch\"utzenberger shows that the Eulerian polynomials have a stronger property, namely $\gamma$-positivity, and a formula of Postnikov, Reiner and Williams does the same for the binomial-Eulerian polynomials.   We obtain  $q$-analogs of both the Foata-Sch\"utzenberger formula and  an alternative to the  Postnikov-Reiner-Williams formula, and we show that these $q$-analogs are specializations of  analogous symmetric function identities.
Algebro-geometric interpretations of these  symmetric function analogs are presented.
\end{abstract}

\date{February 2017; revised June 2018}


\maketitle

\vbox{
\tableofcontents
}

 \section{Introduction} 
 
 In \cite{ChGrKn}, Chung, Graham, and Knuth give several  
 proofs of the following  interesting symmetry involving Eulerian numbers $a_{n,j}$ and binomial coefficients.
 For nonnegative integers $r,s$,
\begin{equation} \label{cgkid}
\sum_{m=1}^{r+s}{{r+s} \choose {m}} a_{m,r-1}=\sum_{m=1}^{r+s}{{r+s} \choose {m}} a_{m,s-1}.
\end{equation}
A  $q$-analog of this identity was  subsequently obtained independently by
Chung and Graham \cite{ChGr} and Han, Lin, and Zeng  \cite{HaLiZe}.

Equation (\ref{cgkid}) is equivalent to palindromicity of the polynomial
$$\tilde A_n(t) = \sum_{j=0}^n \tilde a_{n,j} t^j := 1+ t \sum_{m=1}^n  \binom n m A_m(t),$$ for all $n \ge 0$,
where $A_m(t)$ is the Eulerian polynomial.   We refer to $\tilde A_n(t) $ as a binomial-Eulerian polynomial and $\tilde a_{n,j}$ as a binomial-Eulerian number.   It is well known and easy to prove that the Eulerian polynomials are  palindromic as well.  Hence it is natural to ask whether the binomial-Eulerian polynomials share any other properties with the Eulerian polynomials, such as unimodality.

A polynomial $A(t)=\sum_{j=0}^d a_j t^j \in \R[t]$ is said to be {\em palindromic} if  $a_j = a_{d-j}$ for all $j=0,\dots,d$, and it  is said to be  {\em positive} and
{\em unimodal} if for some $c$
   $$0 \le a_0 \le a_1 \le \cdots \le a_c  \ge \cdots \ge a_{d-1} \ge a_d \ge 0.$$  For example, 
   $A_5(t)=1 + 26 t + 66 t^2 + 26 t^3 +  t^4 $ is  clearly palindromic, positive, and unimodal.
Many important polynomials arising in algebra, combinatorics, and geometry are palindromic, positive and unimodal, see e.g., \cite{st, st2, br}. 

One can easily see that $A(t)$   is  palindromic if and only if there exist  $\gamma_0, \dots, \gamma_{\lfloor {d \over 2} \rfloor} \in \R$
such that
\vspace{-.1in}\begin{equation} \label{gameq} A(t) = \sum_{k=0}^{\lfloor {d \over 2} \rfloor} {\gamma_k} t^k (1+t)^{d-2k}.\end{equation}
The palindromic polynomial $A(t)$ is  said to be {\em $\gamma$-positive} if $\gamma_k \ge 0$ for all $k$.    It is  well known and not difficult to see that 
$\gamma$-positivity implies  unimodality. 

The Eulerian polynomials $A_n(t)$ are $\gamma$-positive as  is evident from the Foata-Sch\"utzenberger  formula \cite[Theorem 5.6]{FoSc1},
\begin{equation} \label{eulereq}A_n(t) = \sum_{k=0}^{\lfloor \frac {n-1} 2 \rfloor} \gamma_{n,k}\, t^k (1+t)^{n-1-2k}, \end{equation} 
where
$\gamma_{n,k} =  |\Gamma_{n,k}|$ 
and $\Gamma_{n,k}$ is the set of permutations $\sigma \in \sg_n$ with 
\begin{itemize}
\item no double descents\footnote{The terminology used here is defined in Section~\ref{prelim}},
\item  no final descent, 
\item $\des(\sigma)=k$.
\end{itemize}
For example $A_5(t) = 1 + 26 t + 66 t^2 + 26 t^3 +  t^4 $ is $\gamma$-positive since $$A_5(t)= 1t^0 (1+t)^4 + 22 t^1 (1+t)^2 + 16 t^2 (1+t)^0 .$$

Recent interest  in $\gamma$-positivity stems from Gal's strengthening \cite{Ga} of  the Charney-Davis conjecture \cite{ChD} 
by asserting that the $h$-polynomial of every  flag simplicial sphere is $\gamma$-positive\footnote{The terminology used here is defined in Section~\ref{GalSec}.}.   Since, as is well known, the Eulerian polynomials  are the $h$-polynomials of dual permutohedra, the Foata-Sch\"utzenberger formula confirms Gal's conjecture for  dual permutohedra.

The permutohedron is an example of a  chordal nestohedron.  In \cite[Section 11.2]{prw}, Postnikov, Reiner, and Williams confirm Gal's conjecture for all dual chordal nestohedra by giving explicit combinatorial formulae for the $\gamma$-coefficients.  Another example of a chordal nestohedron, discussed in \cite[Section 10.4]{prw},  is  the stellohedron, and the  $h$-polynomial of its dual  turns out to be equal to $\tilde A_n(t)$.  It follows that palindromicity of  $\tilde A_n(t)$  is equivalent to the Dehn-Sommerville equations for the dual stellohedron.  

The $\gamma$-positivity formula of Postnikov, Reiner, and Williams  
in the case of the stellohedron says that
\begin{equation} \label{PRWeq} 
\tilde A_n(t)= \sum_{k=0}^{\lfloor \frac {n} 2 \rfloor} \bar \gamma_{n,k}\, t^k (1+t)^{n-2k}, \end{equation} 
where
$\bar\gamma_{n,k}$ is the number of $\sigma \in \sg_{n+1} $ such that $\sigma$ has no double descents,  no final descent, $\sigma(1) < \sigma(2) < \dots < \sigma (m)=n+1 $, for some $m \ge1$, and $\des(\sigma)=k$. 

Here we obtain 
a $\gamma$-positivity formula\footnote{An alternative proof of (\ref{bineulereq}) using poset topological techniques will appear in~\cite{GoWa}.} for $\tilde A_n(t)$ that is somewhat simpler than the Postnikov-Reiner-Williams formula and is similar to the Foata-Schutzenberger formula for $A_n(t)$.  
 For all $n \ge 1$,  
\begin{equation} \label{bineulereq} \tilde A_n(t) = \sum_{k=0}^{\lfloor \frac {n} 2 \rfloor} \tilde \gamma_{n,k}\, t^k (1+t)^{n-2k}, \end{equation} 
where $\tilde\gamma_{n,k} =  |\tilde\Gamma_{n,k}|$ 
and $\tilde \Gamma_{n,k}$ is the set of permutations $\sigma \in \sg_n$ with 
\begin{itemize}
\item no double descents,
\item $\des(\sigma)=k$.
\end{itemize}
(A nice bijection between $\tilde \Gamma_{n,k}$ and the set of permutations  enumerated in the Postnikov-Reiner-Williams formula (\ref{PRWeq}) was obtained by Ellzey \cite{El}.)
Moreover, we present  $q$-analogs  of this $\gamma$-positivity formula (\ref{bineulereq}) and of the  Foata-Sch\"utzenberger formula (\ref{eulereq}), and  observe that they are specializations of  analogous symmetric function identities.
Algebro-geometric interpretations of these  symmetric function analogs are also presented, which suggest an equivariant version of the Gal phenomenon.

The $q$-analogues of the Eulerian numbers and Eulerian polynomials  that we consider   were first examined in previous  work  \cite{ShWa1, ShWa2} of the authors on the joint distribution of the excedance statistic and the major index\footnote{The permutation statistics terminology is defined in Section~\ref{prelim}.}. They are  used in the Chung-Graham, Han-Ling-Zeng $q$-analog of (\ref{cgkid}) mentioned above. The $q$-analog $a_{n,j}(q)$ of the Eulerian number $a_{n,j}$ and the $q$-analog $A_n(q,t)$ of the Eulerian polynomial $A_n(t) $ are polynomials in $\Z[q]$ and $\Z[q][t]$, respectively, defined by
\begin{equation} \label{qEulerDef} A_n(q,t) = \sum_{j=0}^{n-1} a_{n,j}(q) t^j := \sum_{\sigma \in \S_n} q^{\maj(\sigma) -\exc(\sigma)} t^{\exc(\sigma)},\end{equation}
for $n \ge 1$, and $A_n(q,t) := 1$, for $n =0$.
For example,
 \begin{eqnarray*} A_2(q,t) &=& 1+t \\ A_3(q,t) &=& 1+(2+q+q^2) t + t^2 \\
 A_4(q,t) &=& 1+ (3+2q+3q^2+2q^3 + q^4) t +  (3+2q+3q^2+2q^3 + q^4)t^2 + t^3. 
 \end{eqnarray*}
  Another combinatorial description of  $A_n(q,t)$ is given in more recent work  \cite{ShWa3, ShWa4} of the authors.

In \cite{ShWa1,ShWa2}, the authors obtain a $q$-analog of  Euler's formula for the exponential generating function of the Eulerian polynomials,
\begin{equation}\label{qEulereq} \sum_{n \ge 0 } A_n(q,t) \frac {z^n} {[n]_q!} = \frac{\exp_q(z) (1-t)} {\exp_q(tz) - t\exp_q(z)}.\end{equation}
(As is standard, $[n]_q!:=\prod_{j=1}^n[j]_q$, where $[j]_q:=\sum_{i=0}^{j-1}q^i$.  Also, $\exp_q(z) :=\sum_{n \ge 0} \frac {z^n}{[n]_q!}$.)

The $q$-analog $\tilde a_{n,j}(q)$ of the binomial-Eulerian number $\tilde a_{n,j}$ and the $q$-analog $\tilde A_n(q,t)$ of the binomial-Eulerian polynomial $\tilde A_n(t) $ are polynomials in $\Z[q]$ and $\Z[q][t]$, respectively, defined by
$$\tilde A_n(q,t) =  \sum_{j=0}^n \tilde a_{n,j}(q) t^j := 1+ t \sum_{m=1}^n  \binom n m_q A_m(q,t).$$ 
For example,
\begin{eqnarray*} \tilde A_2(q,t) &=& 1+(2+q)t +t^2\\ \tilde A_3(q,t) &=& 1+(3+2q+2q^2) t + (3+2q+2q^2)t^2 + t^3.  \end{eqnarray*}

 The following $q$-analog of (\ref{eulereq}) is  proved in \cite[Equations (1.4) and (6.1)]{LiShWa} and  also appears in Lin and Zeng  \cite{LiZe} (with a different proof).  For $n \ge 1$,
\begin{equation} \label{introqFSeq} A_n(q,t) = \sum_{k=0}^{\lfloor \frac {n-1} 2 \rfloor} \gamma_{n,k}(q) \, t^k (1+t)^{n-1-2k}, \end{equation}  
where
$$\gamma_{n,k}(q) := \sum_{\sigma \in \Gamma_{n,k}} q^{\inv(\sigma) }.$$
Here we give an alternative derivation\footnote{This approach is discussed in earlier work  \cite[Remark 5.5]{ShWa2} of the authors, though the $\gamma_{n,k}(q)$ are not given.} of (\ref{introqFSeq}) and we derive the $q$-analog of (\ref{bineulereq}),
\begin{equation} \label{introqBFSeq} \tilde A_n(q,t) = \sum_{k=0}^{\lfloor \frac {n} 2 \rfloor} \tilde  \gamma_{n,k}(q) \, t^k (1+t)^{n-2k}, \end{equation}  
where 
$$ \tilde \gamma_{n,k}(q) := \sum_{\sigma \in \tilde\Gamma_{n,k}} q^{\inv(\sigma) }.$$

We derive (\ref{introqFSeq}) and (\ref{introqBFSeq}) by specializing  analogous symmetric function identities. 
 These identities involve the symmetric function polynomials $Q_n(\x,t)$ and $\tilde Q_n(\x,t)$, which specialize to $A_n(q,t)$ and $\tilde A_n(q,t)$, respectively, and are  defined as follows.  For $\x=(x_1,x_2,\dots)$, let
 \begin{equation} \label{defQ} \sum_{n \ge 0} Q_n(\x,t) z^n := \frac{(1-t)H(z)} {H(tz) - tH(z)} ,\end{equation} 
 where  $$H(z) := \sum_{n\ge 0} h_n({\x}) z^n,$$ and $h_n(\x)$ is the complete homogeneous symmetric function of degree $n$.  
For $n \ge 0$, let
\begin{equation} \label{defBQ} \tilde Q_{n}(\x,t) :=  h_n(\x) + t \sum_{m=1}^n h_{n-m} (\x) Q_m(\x,t).\end{equation}

For all $n \ge 1$ and $k \ge 0$, let
$$\gamma_{n,k}(\x) := \sum_{D \in  \mathcal H_{n,k}} s_D(\x) ,$$
where $s_D(\x)$ is the skew Schur function of shape $D$ and $\mathcal H_{n,k}$ is  the set of skew hooks of size $n$ for  which $k$ columns have size 2 and  the remaining $n-2k$ columns, including the last column, have size 1.
 From  an interpretation of $Q_n(\x,t)$ due to Gessel \cite{Ge}, we have the  identity,
\begin{equation} \label{introsymFSeq} Q_{n}(\x,t) = \sum_{k=0}^{\lfloor \frac {n-1} 2 \rfloor} \gamma_{n,k}(\x)  \, t^k (1+t)^{n-1-2k},\end{equation}   for all $n \ge 1$.
We use (\ref{introsymFSeq}) to derive the  identity 
\begin{equation} \label{introqBEgammaeq}  \tilde Q_n(\x,t) = \sum_{k=0}^{\lfloor \frac {n} 2 \rfloor} \tilde \gamma_{n,k}(\x) \, t^k (1+t)^{n-2k}, \end{equation}  
where  $$\tilde \gamma_{n,k}(\x) = \sum_{H \in \tilde{\mathcal H}_{n,k}} s_H(\x)$$ and $\tilde{\mathcal H}_{n,k}$ is  the set of skew hooks of size $n$ for  which $k$ columns have size 2 and the remaining $n-2k$ columns have size 1.

 It was shown by Danilov 
 and Jurkiewicz  (see \cite[eq.~(26)]{st2}) 
  that the $h$-polynomial of a simplicial polytope is equal to the Poincar\'e polynomial of the toric variety associated with the polytope.  In \cite{st2} Stanley, using a formula of Procesi \cite{Pr}, gives a representation theoretic interpretation of $Q_n(\x,t)$ involving the toric variety associated with the dual permutohedron.  This and an equivariant version of the hard Lefschetz theorem yield a geometric proof that $Q_n(\x,t)$ is palindromic, Schur-positive and Schur-unimodal.  Here we give an analogous interpretation for $\tilde Q_n(\x,t)$ involving the dual stellohedron.   This leads to the formulation of an equivariant version of  the Gal phenomenon, with the symmetric group actions on the dual permutohedron and the dual stellohedron  exhibiting this phenomenon. 
 
 The paper is organized as follows.  In Section~\ref{prelim}, we recall some basic facts about Eulerian polynomials, permutation statistics, $q$-analogs, and symmetric functions.  The  formulae (\ref{introsymFSeq}) and (\ref{introqBEgammaeq}) are obtained in Section~\ref{SchurSec} and direct proofs of palindromicity and Schur-unimodality of $Q_{n}(\x,t)$ and $\tilde Q_n(\x,t)$ are given.  In Section~\ref{qSec} we show how these formulae specialize to (\ref{introqFSeq}) and (\ref{introqBFSeq}), respectively.   Algebro-geometic interpretations of the results in Section~\ref{SchurSec} are presented in  Section~\ref{GalSec}.  In Section~\ref{DerSec}, we discuss derangement analogs of the results of the previous sections.

\section{Preliminaries} \label{prelim}

 While investigating divergent series in \cite{Eu}, Euler showed that, for each positive integer $n$, there is a monic polynomial $A_n(t) \in \zz[t]$ of degree $n-1$ such that
\[
\sum_{k \geq 0}(k+1)^n t^k =\frac{A_n(t)}{(1-t)^{n+1}}.
\]
Let us write
\[
A_n(t)=\sum_{j=0}^{n-1}a_{n,j}t^j.
\]
The coefficients $a_{n,j}$ of the {\it Eulerian polynomial} $A_n(t)$ are called {\it Eulerian numbers}.  

For a permutation $\s \in \S_n$, the {\it descent set} of $\s$ is
\[
\Des(\s):=\{i \in [n-1]:\s(i)>\s(i+1)\}
\]
and the {\it descent number} of $\s$ is
\[
\des(\s):=|\Des(\s)|.
\]
The fact that
\begin{equation} \label{deseq}
\sum_{\s \in \S_n}t^{\des(\s)}=A_n(t)
\end{equation}
for all $n$ seems to have been observed first by Riordan in \cite{Ri}.  Earlier, MacMahon had shown in \cite[Vol. I, p.186]{MacM} that, with the {\it excedance number} of $\s \in \S_n$ defined as
\[
\exc(\s):=|\{i \in [n-1]:\sigma(i)>i\}|,
\]
the equation
\begin{equation} \label{desexc}
\sum_{\s \in \S_n}t^{\exc(\s)}=\sum_{\s \in \S_n}t^{\des(\s)}
\end{equation}
holds for all $n$.

Recall that the {\it $q$-binomial coefficients} are defined by
\[
{{n} \choose {k}}_q:=\left\{ \begin{array}{ll} \frac{[n]_q!}{[k]_q![n-k]_q!} & 0 \leq k \leq n, \\ 0 & \mbox{otherwise}. \end{array} \right.
\]

There are two additional fundamental permutation statistics, the {\it major index}
\[
\maj(\s):=\sum_{i \in \Des(\s)} i
\]
and  the  {\it inversion number }
$$\inv(\s) := |\{ (i,j) : 1 \le i < j \le n, \s(i) > \s(j) \}| .$$
MacMahon  \cite{MacM} introduced the major index and proved  the first equality in 
\begin{equation*} \label{majeq}
\sum_{\s \in \S_n}q^{\maj(\s)}=[n]_q!=\sum_{\s \in
\S_n}q^{\inv(\s)}
\end{equation*}
after the second equality had been obtained in \cite{Ro} by
Rodrigues.

In \cite{ShWa1,ShWa2}, the authors define a fixed point version of the $q$-Eulerian polynomial, which refines the $q$-Eulerian polynomial given in   (\ref{qEulerDef}). For $n \ge 1$, let
$$A_n(q,t,r) :=  \sum_{\sigma \in \S_n} q^{\maj(\sigma) -\exc(\sigma)} t^{\exc(\sigma)} r^{\fix(\sigma)},$$
where $\fix(\sigma)$ is the number of fixed points of $\sigma$, and let $A_0(q,t,r) :=1$.  So $A_n(q,t,1) = A_n(q,t)$ for all $n \ge 0$. In \cite{ShWa1,ShWa2},  the refinement of (\ref{qEulereq}),
\begin{equation}\label{qfixEulereq} \sum_{n \ge 0 } A_n(q,t,r) \frac {z^n} {[n]_q!} = \frac{\exp_q(rz) (1-t)} {\exp_q(tz) - t\exp_q(z)}\end{equation}
is derived.

As mentioned in the introduction,  the Foata-Schutzenberger formula  (\ref{eulereq}) establishes $\gamma$-positivity of the Eulerian polynomials and the Postnikov-Reiner-Williams formula (\ref{PRWeq}) establishes $\gamma$-positivity of the binomial-Eulerian polynomials.  We now give precise definitions of the terminology used in these formulas.
 We say $\sigma \in \sg_n$ has  
\begin{itemize}
\item a {\it double descent}  if there exists $i \in [n-2]$ such that $\sigma(i) > \sigma(i+1) > \sigma(i+2)$
\item an {\it initial descent} if $\sigma(1) > \sigma(2)$ 
\item  a {\it final descent} if $\sigma(n-1) > \sigma(n)$.
\end{itemize} 

We say that a polynomial $f(q) \in \R[q] $ is $q$-positive if its  coefficients are nonnegative. Given polynomials $f(q),g(q) \in \R[q]$ we say that $f(q) \le_q g(q)$ if $g(q) - f(q)$ is $q$-positive.  
More generally, let $R$ be an algebra over $\R$ with basis $b$.  An element $s\in R$ is said to be $b$-{\em positive} if the expansion of $s$ in the basis $b$ has nonnegative coefficients.  Given $r,s \in R$, we say that $r \le_b s$ if $s-r$ is $b$-positive. 

The $\R$-algebras  considered in this paper are  $R$= $\R$,  $\R[q]$, and  the algebra $\Lambda$ of symmetric functions over $\R$.
 If $R=\R$ and $b= \{1\}$ then $b$-positive is the same as positive and $<_b$ is the usual numerical $<$ relation.  If $R = \R[q]$ and $b = 
\{q^i: i \in \N\}$ then $b$-positive is  what we called $q$-positive above and $<_b$ is the same as $<_q$.  For $R= \Lambda$, we consider the  basis of Schur functions $\{s_\lambda(\x) : \lambda  \in \cup_{n\ge 0} \mbox{ Par}(n) \} $ and the basis of complete homogeneous symmetric functions $\{h_\lambda(\x): \lambda \in \cup_{n\ge 0} \mbox{ Par}(n)\} $, where $\mbox{ Par}(n)$ is the set of partitions of $n$.  It is a basic fact that $h$-positive implies Schur-positive (see for example \cite[Proposition 7.18.7]{St}).

\begin{definition}\label{paldef} Let $R$ be an $\R$-algebra with basis $b$.   We say that a polynomial
$A(t):=a_0 + a_1 t +\cdots  + a_{n} t^{n}\in R[t]$ is 
\begin{itemize}
\item $b$-{\em positive} if each coefficient $a_i$ is $b$-positive  
\item  $b$-{\em unimodal} if for some $c$,
$$a_0 \le_b a_1\le_b \cdots \le_b a_c \ge_b a_{c+1} \ge_b a_{c+2} \ge_b \cdots \ge_b a_n,$$ 
\item {\em palindromic} with center of symmetry $\frac n 2$ if $a_j = a_{n-j}$  for $0 \le j \le n$,
\item {\em $b$-$\gamma$-positive} if there exist $b$-positive $\gamma_0,\dots,  \gamma_{\lfloor {d \over 2} \rfloor} \in R$ such that 
$$A(t) = \sum_{k=0}^{\lfloor {d \over 2} \rfloor} {\gamma_k}  t^k (1+t)^{d-2k}.$$  
\end{itemize}
\end{definition} 

The following results are well known, at least in the case that $R= \R$ (see \cite[Appendix B]{ShWa4}).

\begin{prop}[see {\cite[Proposition 1]{st2}}] \label{tooluni} Let $R$ be an $\R$-algebra with basis $b$. Let $A(t)$ and $B(t)$ be palindromic, $b$-positive,  $b$-unimodal  polynomials in $R[t]$ with respective centers of symmetry $c_A$ and $c_B$.  Then
\begin{enumerate}
\item $A(t) B(t) $ is palindromic, $b$-positive,  $b$-unimodal  with center of symmetry $c_A +c_B$.
\item If $c_A=c_B$ then $A(t) +B(t)$ is palindromic,  $b$-positive, $b$-unimodal with center of symmetry $c_A$.
\end{enumerate}
\end{prop} 

\begin{cor} \label{UniProp}    If $A(t) \in R[t]$ is $b$-$\gamma$-positive then $A(t)$ is palindromic, $b$-positive, and $b$-unimodal.
\end{cor}

\section{Schur-$\gamma$-positivity} \label{SchurSec}

  In this section we establish Schur-$\gamma$-positivity of  the symmetric function analogs $Q_n(\x,t)$ and $\tilde Q_n(\x,t)$  given in (\ref{defQ}) and (\ref{defBQ}), and we present  combinatorial formulae for the $\gamma$-coefficients.
We also present direct proofs of palindromicity, Schur-positivity, and Schur-unimodality, which don't rely on Schur-$\gamma$-positivity.

 It is an easy consequence of the following  result of Gessel that $Q_n(\x,t)$ is Schur-$\gamma$-positive. Let $\pp_n$ be the set of words of length $n$ over the alphabet of positive integers $\pp$.
Given a word $w \in \pp_n$, we let $w_i$ denote its $i$th letter.  That is, $w = w_1w_2 \dots w_n$.  Just as for permutations, let $\des(w)$ equal the number of $i \in [n-1]$ such that $w_i > w_{i+1}$. A word $w$ is said to have a double descent if there exists an $i \in [n-2]$ such that $w_i>w_{i+1}>w_{i+2}$. Let $NDD_n$ be the set of words in $\pp_n$ with no double descents.  For $w \in \pp_n$, let $\x_w:=x_{w_1}x_{w_2}\dots x_{w_n}$.

 \begin{thm}[Gessel \cite{Ge}, see {\cite[Theorem 7.3]{ShWa2}}] \begin{equation} \label{ges} 1+\sum_{n \ge 1} z^n \sum_{\scriptsize \begin{array}{c} w \in NDD_n \\ w_{n-1} \le w_n\end{array}} \,{\bf x}_w \, t^{\des(w)} (1+t) ^{n-1-2\des(w)} = {(1-t) H(z) \over H(zt) - t H(z)} ,\end{equation}
 where  $w_0=0$.
\end{thm}

The symmetric function polynomial $Q_n(\x,t)$  defined in (\ref{defQ}) can now be given an explicit expansion  which establishes Schur-$\gamma$-positivity.  The $\gamma$-coefficients are described in terms of  hook shaped skew Schur functions. A {\it skew hook} is a connected skew diagram with no $2 \times 2$ square.  Let $\mathcal H_{n,k}$ be  the set of skew hooks of size $n$ for  which $k$ columns have size 2 and  the remaining $n-2k$ columns, including the last column, have size 1. For example, 
$$ {\tableau{&& && & {} & {} \\ && {}&{} & {} & {} \\ {} & {} & {} }} \quad  \in \quad \mathcal H_{9,2}.$$

\begin{cor} \label{gescor}  Let 
\begin{equation} \label{symAEcoef}\gamma_{n,k}(\x) := \sum_{D \in  \mathcal H_{n,k}} s_D(\x) ,\end{equation} 
where $s_D(\x)$ is the skew Schur function of shape $D$.
Then \begin{equation} \label{symFSeq}Q_n(\x,t) = \sum_{k=0}^{\lfloor \frac {n-1} 2 \rfloor} \gamma_{n,k}(\x)  \, t^k (1+t)^{n-1-2k}. \end{equation}  
Consequently the polynomial  $Q_n(\x,t)$ is Schur-$\gamma$-positive. 
\end{cor} 

\begin{proof}  By (\ref{ges}), for $n \ge 1$,
$$ Q_n(\x,t) = \sum_{\scriptsize \begin{array}{c} w \in NDD_n \\ w_{n-1} \le w_n\end{array}}  \,{\bf x}_w \, t^{\des(w)} (1+t) ^{n-1-2\des(w)}.$$  Note that the semistandard tableaux of hook shape in $\mathcal H_{n,k}$ correspond bijectively to words $w \in NDD_n$ with $w_{n-1} \le w_n$ and with $k$ descents.  Indeed by reading the entries of such a semistandard tableau from southwest to northeast, one gets such a word.  
For example, the semistandard tableau
$$ {\tableau{&& && & {2} & {8} \\ && {1}&{1} & {8} & {9} \\ {2} & {5} & {5} }} $$
corresponds to the word $255118928 \in NDD_9$, which has $2$ descents.
It follows that
$$\sum_{\scriptsize \begin{array}{c} w \in NDD_n \\ w_{n-1} \le w_n\\ \des(w) =k \end{array}}  \x_w = \sum_{D \in \mathcal H_{n,k}} s_D(\x).$$


The consequence follows from the fact that skew Schur functions are Schur-positive.
\end{proof}

Next we derive an analogous Schur-$\gamma$-positivity result for    $\tilde Q_{n}(\x,t)$, which was defined in (\ref{defBQ}).
We begin with a generating function formula.

\begin{prop} \label{newbinomprop}
\begin{equation} \label{genfuntildeQ} \sum_{n \ge 0} \tilde Q_n(\x,t) z^n =  {(1-t) H(z)H(tz) \over H(tz) - t H(z)}.\end{equation}
Equivalently, for all $n \ge 0$,
\begin{equation} \label{newbinomeq} \tilde Q_{n}(\x,t) =   \sum_{m=0}^n  h_{n-m}(\x) Q_m(\x,t) t^{n-m} .\end{equation}
\end{prop}

\begin{proof} By the definitions (\ref{defBQ}) and (\ref{defQ}),
\begin{eqnarray*} \sum_{n \ge 0} \tilde Q_{n}(\x,t) z^n  &=& H(z)(1+t\sum_{n\ge 1} Q_n(\x,t) z^n) \\ &=& H(z) (1+ t( {(1-t) H(z)  \over H(tz) -tH(z)} -1)  )
\\ 
&=& H(z) \frac{H(tz)(1-t) }{H(tz)-tH(z)}.
\end{eqnarray*}
\end{proof}

Let $\tilde{\mathcal H}_{n,k}$ be  the set of skew hooks of size $n$ for  which $k$ columns have size 2 and the remaining $n-2k$ columns have size 1.

\begin{thm} \label{gesBcor}   Let
\begin{equation} \label{symBEcoef} \tilde \gamma_{n,k}(\x) := \sum_{D \in  \tilde {\mathcal H}_{n,k}} s_D(\x) ,\end{equation}
where $s_D(\x)$ is the skew Schur function of shape $D$.
Then \begin{equation} \label{BEgammaeq} \tilde Q_n(\x,t) = \sum_{k=0}^{\lfloor \frac {n} 2 \rfloor} \tilde \gamma_{n,k}(\x) \, t^k (1+t)^{n-2k}. \end{equation}  
Consequently the polynomial $\tilde Q_n(\x,t)$ is Schur-$\gamma$-positive.
\end{thm}

\begin{proof}

For $n \ge 1$, let $$W_n(\x,t):=   \sum_{\scriptsize \begin{array}{c} w \in NDD_n \\ w_{n-1} \le w_n\end{array}} \,{\bf x}_w \, t^{\des(w)} (1+t) ^{n-1-2\des(w)}$$ and let 
 $$\tilde W_n(\x,t):=   \sum_{ w \in NDD_n } \,{\bf x}_w \, t^{\des(w)} (1+t) ^{n-2\des(w)}.$$
 
By (\ref{ges}), we have 
\begin{equation} \label{WQeq} W_n(\x,t) = Q_n(\x,t).\end{equation}  
We will show that
 \begin{equation} \label{tildeWeq} \tilde W_n(\x,t) = h_n(\x) t^n+  \sum_{m =1}^n h_{n-m}(\x)W_m(\x,t)  t^{n-m}.\end{equation}
It follows from this, (\ref{newbinomeq}), and (\ref{WQeq})  that    $\tilde W_n(\x,t) = \tilde Q_n(\x,t)$.   
This is equivalent to the desired result since the semistandard  tableaux of skew hook shape in $\tilde{\mathcal H}_{n,k}$ correspond bijectively to words  in $NDD_n$ with $k$ descents.

Let $I_n $ be the set $\{\alpha \in \PP_n: \alpha_1 \le \dots \le \alpha_n\}$ of weakly increasing words of length $n$.  The right side of (\ref{tildeWeq}) equals
\begin{eqnarray*} & &\sum_{u\in I_n}t^{n} \x_u + \sum_{m=1}^n \sum_{\scriptsize \begin{array}{c} w \in NDD_m \\ w_{m-1} \le w_m\end{array}}  t^{\des(w)} (1+t)^{m-1-2\des(w)} \x_{w}
\sum_{u \in I_{n-m}}t^{n-m} \x_u  \\
&=& \sum_{u \in I_n}t^{n}\x_u +
 \sum_{m=1}^n \ \sum_{\scriptsize \begin{array}{c} w \in NDD_m \\ w_{m-1} \le w_m \\ u\in I_{n-m}\end{array}}  t^{\des(w)+(n-m)} (1+t)^{m-1-2\des(w)} \x_{w\cdot u} 
\end{eqnarray*}
where $w \cdot u$ denotes concatenation of words $w$ and $u$.

For $v \in \PP_n$ we seek the coefficient of $\x_v$.  Note that the coefficient is 0 if $v$ has a double descent.  For $v \in NDD_n$, let $j$ be the smallest integer such that $v_j \le v_{j+1} \le \dots \le v_n$.  So $j-1$ is either $0$ (when $v$ is weakly increasing) or the position of the last descent.  Each $m \in \{j+1, \dots, n\}$ determines a  decomposition of $v$ into $w \cdot u$, where $w \in NDD_m$, $w_{m-1 } \le w_m$  and  $u \in I_{n-m}$.   Note that $\des(v) = \des(w)$.

The only other value of $m$ that determines a decomposition of $v$ into $w \cdot u$ for which $w \in NDD_m$, $w_{m-1 } \le w_m$  and  $u \in I_{n-m}$, is $m = j-1$.  In this case, if $j-1 >0$ we have $\des(v) = \des(w) +1$.
It follows that if $j > 1$, the coefficient $c_v$ of $\x_v$  is given by
$$c_v =  t^{\des(v)+n-j}  (1+t)^{j-2\des(v)} + \sum_{m=j+1}^n t^{\des(v)+n-m} (1+t)^{m-1-2\des(v)}.$$
We have \begin{eqnarray}\label{comput} \nonumber \sum_{m=j+1}^n t^{\des(v)+n-m} (1+t)^{m-1-2\des(v)}
 &=& t^{\des(v)+n-j-1}  (1+t)^{j-2\des(v)} \sum_{k=0}^{n-j-1} \left(\frac{1+t}{t}\right)^{k}
 \\ \nonumber &=& t^{\des(v)+n-j}  (1+t)^{j-2\des(v)} \left ( \left ( \frac{1+t} t \right )^{n-j} -1\right)
 \\ &=& t^{\des(v)} (1+t)^{n-2\des(v)} - t^{\des(v)+n-j}  (1+t)^{j-2\des(v)},
\end{eqnarray}  
from which we conclude that $c_v = t^{\des(v)} (1+t)^{n-2\des(v)}$.

Now if $j=1$ then $v$ is a weakly increasing word and  the coefficient of $\x_v$ is given by
$$c_v =  t^n + \sum_{m=1}^n t^{n-m} (1+t)^{m-1} .$$  A simple computation shows that the summation is equal to 
$(1+t)^n - t^n$.  Hence $c_v = (1+t)^n = t^{\des(v)}(1+t)^{n-2\des(v)}$, as in the previous case.
We have therefore shown that the right hand side of (\ref{tildeWeq}) is equal to $$ \sum_{v \in NDD_n} t^{\des(v)}(1+t)^{n-2\des(v)} \x_v,$$ which by  definition is the left side of (\ref{tildeWeq}).   \end{proof}

\begin{remark}  It was pointed out to us by Gonz\'alez D'Le\'on that another identity of Gessel  \cite[Theorem 4.2]{Ge1} can be used to give an alternative proof of Theorem~\ref{gesBcor}, or equivalently of $\tilde Q_n(\x,t) = \tilde W_n(\x,t)$.   By inverting (\ref{tildeWeq}), one can conclude from this that $ Q_n(\x,t) = W_n(\x,t)$, which is equivalent to Gessel's unpublished result  (\ref{ges}).  Hence \cite[Theorem 4.2]{Ge1} can be used to prove (\ref{ges}).  Gessel \cite{Ge} has  a more direct proof of (\ref{ges}) however.
\end{remark}

The following result for $Q_n(\x,t)$ was first obtained by Stanley \cite{st2} from the algebro-geometric interpretation of $Q_n(\x,t)$ given in (\ref{StanProEq}).  

\begin{cor} \label{SchurUniCor} For all $n \ge 0$, the symmetric function polynomials $Q_n(\x,t)$ and $\tilde Q_n(\x,t)$ are palindromic, Schur-positive, and Schur-unimodal.
\end{cor}

\begin{proof} Use Corollary~\ref{UniProp}.
\end{proof}

A stronger result  for $Q_n(\x,t)$ was proved by Stembridge \cite{Ste}, namely $h$-positivity and $h$-unimodality of $Q_n(\x,t)$. 
A simpler proof of this result given in  \cite[Corollary C.5]{ShWa4}
relies on the formula 
\begin{equation} \label{newFormQ} \sum_{n\ge 0} Q_n(\x,t) z^n = 1+ \frac{\sum_{n\ge 1} [n]_t h_n z^n} { 1-t \sum_{n\ge 2} [n-1]_t h_n z^n}\end{equation}  and Proposition~\ref{tooluni}. Here we give an alternative proof of Corollary~\ref{SchurUniCor} for $\tilde Q_n(\x,t)$ that  does not rely on Theorem~\ref{gesBcor}.

\begin{proof}[Alternative proof of Corollary~\ref{SchurUniCor} for  $\tilde Q_n(\x,t)$]  Let $Q^0_n(\x,t)$ be defined by 
$$\sum_{n\ge 0} Q^0_n(\x,t) z^n = {1-t \over H(tz) - t H(z)}= {1 \over 1-t\sum_{n\ge 2} [n-1]_t h_nz^n}.$$  It follows from Proposition~\ref{tooluni}  that $Q^0_n(\x,t)$ is palindromic, $h$-positive and $h$-unimodal with center of symmetry ${n \over 2}$.  By Proposition~\ref{newbinomprop}, 
\begin{equation} \label{Q0eq} \tilde Q_{n}(\x,t) =\sum_{k\ge 0}\left( \sum_{j=0}^k t^j h_j h_{k-j}\right) Q^0_{n-k}(\x,t). \end{equation}  

It is easy to see that $\sum_{j=0}^k t^j h_j h_{k-j}$ is palindromic with center of symmetry $\frac {k}2$.  It is clearly $h$-positive, which implies that it is Schur-positive. We claim that it is also  Schur-unimodal.  If $j \le k-j$ then by Pieri's rule $h_j h_{k-j} = \sum_{i=0}^j s_{k-i,i}$.  From this we can see that $\sum_{j=0}^k t^j h_j h_{k-j}$ is Schur-unimodal.   By Proposition~\ref{tooluni},  we have that
$\left(\sum_{j=0}^k t^j h_j h_{k-j} \right) Q^0_{n-k}(\x,t)$ is palindromic, Schur-positive, and Schur unimodal with center of symmetry equal to ${k \over 2} + {n-k \over 2} = {n \over 2}$.  Again by Proposition~\ref{tooluni}, we can conclude from (\ref{Q0eq}) that  $\tilde Q_n(\x,t)$ is palindromic, Schur-positive, and Schur-unimodal with center of symmetry ${n \over 2}$.
\end{proof}

\section{$q$-$\gamma$-positivity of the $q$-Eulerian and $q$-binomial-Eulerian polynomials} \label{qSec}

It this section we use the results of the previous section to prove that the $q$-Eulerian polynomials 
$$A_n(q,t) :=  \sum_{\sigma \in \S_n} q^{\maj(\sigma) -\exc(\sigma)} t^{\exc(\sigma)}$$ and $q$-binomial-Eulerian polynomials 
$$\tilde A_n(q,t) :=   1+ t \sum_{m=1}^n  \binom n m_q A_m(q,t)$$ 
 are $q$-$\gamma$-positive.
 
 From any symmetric function $G(x_1,x_2,\ldots)$ one obtains a power series in a single variable $q$ by the {\it stable principal specialization}, in which each $x_i$ is replaced by $q^{i-1}$.  Let
\[
\ps_q(G):=G(1,q,q^2,\ldots).
\]
This definition can be extended to polynomials in $\Lambda[t]$ by defining,
$$\ps_q(\sum_{i=o}^d G_i(\x) t^i) := \sum_{i=0}^d \ps_q(G_i(\x)) t^i.$$

 Let   $SYT_D$ denote the set of standard Young tableaux of skew shape $D$.  For $T \in SYT_D$ (written in English notation), let $\Des(T)$ be the set of  entries $i $ of $T$ for which $i$ is in a higher row than $i+1$, and let $\maj(T) = \sum_{i \in \Des(T)} i$.  It is well known  (see \cite[Proposition 7.19.11]{St}) that  
\begin{equation} \label{psSchur} \ps_q(s_D) = \frac {\sum_{T \in SYT_D} q^{\maj(T)}}{(1-q) \dots (1-q^n) },\end{equation} where $n$ is the number of cells of $D$.  It follows from this (and is easy to see directly) that
 $$\ps_q(h_n) = \frac{1} {(1-q) \dots (1-q^n) }.$$

By taking stable principal specialization of both sides of (\ref{defQ}), one can see that the following result is equivalent to (\ref{qEulereq}).   In fact, in \cite{ShWa2} this result was used to prove  (\ref{qEulereq}).

\begin{thm}[Shareshian and Wachs \cite{ShWa2}] \label{SWps} For all $n \ge 0$,
$$\ps_q(Q_n(\x,t)) = \frac {A_n(q,t)}{(1-q) \dots (1-q^{n})}$$
\end{thm}

An analogous result holds for the $q$-binomial-Eulerian polynomials.
\begin{cor} \label{TildeSWps} For all $n \ge 0$,
$$\ps_q(\tilde Q_n(\x,t)) = \frac {\tilde A_n(q,t)}{(1-q) \dots (1-q^{n})}.$$
\end{cor}

\begin{proof}  Starting with the definition of $\tilde Q_n(\x,t)$ given in (\ref{defBQ}), we have
\begin{eqnarray*} \ps_q(\tilde Q_n(\x,t)) &=& \ps_q(h_n) + t \sum_{m=1}^n  \ps_q(h_{n-m}) {\ps_q(Q_m(\x,t))}  \\ &=& \frac{1} {\prod_{i=1}^n (1-q^i)} +t \sum_{m=1}^n \frac {A_m(q,t)} {\prod_{i=1}^m (1-q^i) \prod_{i=1}^{n-m} (1-q^i)}  
\\  &=& \frac{1+ t \sum_{m=1}^n  \binom n m_q A_m(q,t)} {\prod_{i=1}^n (1-q^i)} 
\\ &=& \frac{\tilde A_n(q,t)} {\prod_{i=1}^n (1-q^i)}  ,\end{eqnarray*}
with the second equality following from Theorem~\ref{SWps}.
\end{proof}

By taking the stable principal  specialization of both sides of (\ref{genfuntildeQ}), one gets the following result. The consequences follow from (\ref{qEulereq}) and (\ref{qfixEulereq}), respectively. 
\begin{prop} \label{qnewbinomprop}
$$ \sum_{n \ge 0} \tilde A_n(q,t) \frac{z^n}{[q]_n!} = \frac{(1-t) \exp_q(z) \exp_q(tz)}{ \exp_q(tz) - t\exp_q(z)} .$$
Consequently $$\tilde A_n(q,t) = \sum_{m=0}^n \binom n m_q A_m(q,t) t^{n-m}$$
and 
$$\tilde A_n(q,t) = \sum_{m=0}^n \binom n m_q A_m(q,t,t). $$
\end{prop}

  In \cite[Remark 5.5]{ShWa2}, the authors mention that (\ref{ges}) can be used to establish $q$-$\gamma$-positivity of $A_n(q,t)$.  Now we carry this out by using (\ref{symFSeq}) to obtain the $\gamma$-coefficients.
  The following result is  proved in \cite[Equations (1.4) and (6.1)]{LiShWa} without the use of  (\ref{ges}).  
  
\begin{thm} \label{qFSth} Let $\Gamma_{n,k}$ be the set of permutations $\sigma \in \sg_n$ with no double descents, no final descent, and with $\des(\sigma)=k$, and let 
$$\gamma_{n,k}(q) := \sum_{\sigma \in \Gamma_{n,k}} q^{\inv(\sigma) }  \,\,\,\,\,\,(=\sum_{\sigma \in \Gamma_{n,k}} q^{\maj(\sigma^{-1}) }).$$
Then \begin{equation} \label{qFSeq} A_n(q,t) = \sum_{k=0}^{\lfloor \frac {n-1} 2 \rfloor} \gamma_{n,k}(q) \, t^k (1+t)^{n-1-2k}. \end{equation}  
Consequently the $q$-Eulerian polynomials $A_n(q,t)$ are $q$-$\gamma$-positive.
\end{thm}

\begin{proof}  By applying stable principal specialization to both sides of (\ref{symFSeq}) we have
\begin{equation} \label{psSymFSeq} \ps_q(Q_n(\x,t)) =  \sum_{k=0}^{\lfloor \frac {n-1} 2 \rfloor} \ps_q(\gamma_{n,k}(\x))  \, t^k (1+t)^{n-1-2k}. \end{equation}

By   (\ref{symAEcoef}) and (\ref{psSchur}), we have
\begin{eqnarray} \label{psgammaeq} \ps_q(\gamma_{n,k}(\x)) &=& \sum_{D \in  \mathcal H_{n,k}} \ps_q(s_D(\x)) \\ \nonumber
&=& \sum_{D \in  \mathcal H_{n,k}} \frac {\sum_{T \in SYT_D} q^{\maj(T)}}{(1-q) \dots (1-q^n) }
.\end{eqnarray}

 If  $D$ is a skew hook then $SYT_D$ corresponds bijectively to the set of permutations in $\sg_n$ with a fixed descent set determined by $D$.  Indeed, by reading the entries of $T \in SYT_D$ from southwest to northeast, one gets a permutation $ \varphi(T) \in \sg_n$.  
Descents are encountered whenever one goes up a column.   So  $\Des(\varphi(T)) $ equals the set of all $i \in [n-1]$ such that the $i$th cell of $D$ (ordered from southwest to northeast) is directly below the $(i+1)$st cell of  $D$. It follows that  if $D \in \mathcal H_{n,k}$ 
and $T \in SYT_D$ then $\varphi(T) \in \Gamma_{n,k}$.

Note also
that for $T \in SYT_D$, $\Des(T) = \Des(\varphi(T)^{-1})$.  We can now conclude that
\begin{equation}\label{skewhookeq} \sum_{D \in \mathcal H_{n,k}} \sum_{T \in SYT_D} q^{\maj(T)} = \sum_{\sigma \in \Gamma_{n,k}} q^{\maj(\sigma^{-1})}.\end{equation}

For each $J \subset [n-1]$, the {\em descent class} of $J$ is the set $\{\sigma \in \sg_n : \Des(\sigma) = J\}$.   Note that 
$\Gamma_{n,k}$ is a union of descent classes. By  the Foata-Sch\"utzenberger result   \cite[Theorem 1]{FoSc2}  that $\inv(\sigma)$ and $\maj(\sigma^{-1})$  are equidistributed on descent classes, we have   
 $$ \sum_{\sigma \in \Gamma_{n,k}} q^{\maj(\sigma^{-1})}= \sum_{\sigma \in \Gamma_{n,k}} q^{\inv(\sigma)}.$$
Combining this with (\ref{skewhookeq}) and substituting  in (\ref{psgammaeq}) results in
$$ \ps_q(\gamma_{n,k}(\x)) =
\frac{\sum_{\sigma \in \Gamma_{n,k}} q^{\inv(\sigma)}}{(1-q) \dots (1-q^n) }.$$
It follows that the right side of (\ref{psSymFSeq}) equals 
$$\frac{ \sum_{k=0}^{\lfloor \frac {n-1} 2 \rfloor} \sum_{\sigma \in \Gamma_{n,k}} q^{\inv(\sigma)} \, t^k (1+t)^{n-1-2k}} {(1-q) \dots (1-q^n)},$$
while, by  Theorem~\ref{SWps}, the left side equals
$$\frac {A_n(q,t)}{(1-q) \dots (1-q^{n})},$$
thereby completing the proof.
\end{proof}

By  taking the stable principal specialization of both sides of equation (\ref{BEgammaeq}) and using an argument  analogous  to the proof of Theorem~\ref{qFSth}, we obtain the following result. 
\begin{thm} \label{qBFSth} Let $\tilde \Gamma_{n,k}$ be the set of permutations $\sigma \in \sg_n$ with no double descents and with $\des(\sigma)=k$, and let 
\begin{equation} \label{qBEcoef} \tilde \gamma_{n,k}(q) := \sum_{\sigma \in \tilde\Gamma_{n,k}} q^{\inv(\sigma) } \,\,\,\,\,\,(=  \sum_{\sigma \in \tilde\Gamma_{n,k}} q^{\maj(\sigma^{-1}) }).\end{equation}
Then \begin{equation} \label{qBEgammaeq} \tilde A_n(q,t) = \sum_{k=0}^{\lfloor \frac {n} 2 \rfloor} \tilde  \gamma_{n,k}(q) \, t^k (1+t)^{n-2k}. \end{equation}  
Consequently, the $q$-binomial-Eulerian polynomials $\tilde A_n(q,t)$ are $q$-$\gamma$-positive.
\end{thm}

The following result for $A_n(q,t)$ was first obtained by the authors in \cite{ShWa2}.

\begin{cor} \label{quniCor} For all $n \ge 0$, the polynomials $A_n(q,t)$ and $\tilde A_n(q,t)$ are palindromic and $q$-unimodal.
\end{cor}

Just as for Corollary~\ref{SchurUniCor}, an alternative proof of Corollary~\ref{quniCor} can be given which doesn't make use of Theorems~\ref{qFSth} and~\ref{qBFSth}.  For $A_n(q,t)$ a simple proof is given in Appendix C.1 of \cite{ShWa4} by using the  formula 
$$ 1+\sum_{n \geq 1}A_n(q,t)\frac{z^n}{[n]_q!} = 1 + \frac {\sum_{n \ge 1} [n]_t \frac{z^n}{[n]_q!}} {1-t\sum_{n\ge 2} [n-1]_t \frac{z^n}{[n]_q!} }
$$
 obtained by manipulating (\ref{qEulereq}).

 \begin{proof}[Alternative proof of Corollary~\ref{quniCor} for $\tilde A_n(q,t)$] 
  By (\ref{qfixEulereq}) and Proposition~\ref{qnewbinomprop},
 $$\tilde A_n(q,t)  = \sum_{k \ge 0} \left (\sum_{j=0}^k \binom  k j_q  t^j \right) A_{n-k}(q,t,0).$$
 Since 
$$ \sum_{n \ge 0 } A_n(q,t,0) \frac {z^n} {[n]_q!} = \frac{ (1-t)} {\exp_q(tz) - t\exp_q(z)} = \frac 1 {1-t\sum_{n\ge 2} [n-1]_t \frac {z^n}{[n]_q!} },$$
it follows from Proposition~\ref{tooluni} that 
$A _n(q,t,0)$ is palindromic and $q$-unimodal with center of symmetry $\frac n 2$.  It is well known that 
 $\sum_{j=0}^k \binom  k j_q  t^j$ is palindromic and $q$-unimodal with center of symmetry $\frac k 2$. Note that this  follows from taking the stable principal specialization of $\sum_{j=0}^k h_j h_{k-j} t^j$, which we observed to be Schur-unimodal in the alternative proof of Corollary~\ref{SchurUniCor}.  By Proposition~\ref{tooluni},  $\tilde A_n(q,t)$ is a sum of palindromic, $q$-positive, $q$-unimodal polynomials with center of symmetry $\frac k 2 + \frac {n-k} 2$.  It therefore follows again from Proposition~\ref{tooluni} that $\tilde A_n(q,t)$ is palindromic and $q$-unimodal.
 \end{proof}

Note that palindromicity of $\tilde A_n(q,t)$ is equivalent to the following $q$-analog of (\ref{cgkid}).
\begin{cor}[Chung-Graham \cite{ChGr} and Han-Lin-Zeng  \cite{HaLiZe}]  For positive integers $r,s$,
$$
\sum_{m=1}^{r+s}{{r+s} \choose {m}}_q a_{m,r-1}(q)=\sum_{m=1}^{r+s}{{r+s} \choose {m}}_q a_{m,s-1}(q).
$$
\end{cor}

A symmetric function analog is given by the following result, which is equivalent to palindromicity of $\tilde Q_n(\x,t)$. (A more general result appears as Theorem 2 in the preprint \cite{Lin} of Z. Lin.) 

\begin{cor} For positive integers $r,s$,
$$\sum_{m=1}^{r+s} h_{r+s-m} Q_{m,r-1} =  \sum_{m=1}^{r+s} h_{r+s-m} Q_{m,s-1}.$$
\end{cor}

\section{Geometric interpretation: equivariant Gal phenomenon} \label{GalSec}

In this section we will present  interpretations of results in Section~\ref{SchurSec} using geometry and representation theory.  The idea behind such interpretations was, to our knowledge, first employed by Stanley, and is discussed in \cite{st2}. 

Herein, a {\em polytope} is the convex hull of a finite set of points in some $\R^{d}$.  A polytope is {\em simplicial} if every proper face is a simplex.  Let $P$ be a  $d$-dimensional simplicial  polytope. Associated with  $P$ is the  $h$-polynomial  defined by 
$$h_P(t) := \sum_{j=0}^d f_{j-1} (t-1)^{d-j} ,
$$
where $f_i$ is the number of faces of $P$ of dimension $i$.  
 It is well known that the $h$-polynomial
of every simplicial polytope is palindromic and unimodal.  Indeed, palindromicity is equivalent to the Dehn-Sommerville equations, and unimodality was proved by Stanley \cite{st1}  as part of the  g-Theorem of Billera, Lee and Stanley (see e.g.,  \cite{st,Bi}).

A simplicial complex is said to be {\em flag} if it is the clique complex of its 1-skeleton; that is, its faces are the cliques of its 1-skeleton.  Examples of flag simplicial complexes include  barycentric subdivisions of  simplicial complexes, or more generally  order complexes of  posets.  Gal formulated the following strengthening of the long standing Charney-Davis conjecture \cite{ChD}.
\begin{con}[Gal \cite{Ga}]    If $P$ is a flag simplicial polytope (or more generally a flag simplicial sphere) then $h_P(t)$ is $\gamma$-positive.
\end{con}  
Gal's  conjecture has been proved for certain special classes and examples; see \cite[Section 10.8]{Pe}.   One such example is the dual   of the permutohedron.  The permutohedron $P_n$  is the convex hull of the set $\{(\sigma(1),\dots, \sigma(n)) : \sigma \in \sg_n\}$.   The dual permutohedron $P_n^*$ is combinatorially equivalent to the barycentric subdivision of the boundary of the $(n-1)$-simplex.  Clearly $P_n^*$ is a flag simplicial polytope.  It is well known that
$$h_{P_n^*}(t) = A_n(t).$$ Hence by (\ref{eulereq}), $h_{P_n^*}(t)$ is $\gamma$-positive.

We will say that a flag simplicial polytope $P$  {\it exhibits Gal's phenomenon} if  $h_P(t)$ is $\gamma$-positive.  So $P_n^*$ exhibits Gal's phenomenon.  The permutohedron and another polytope called the stellohedron belong to a class of  polytopes called chordal nestohedra.  
In \cite[Section 11.2]{prw} Postnikov, Reiner, and Williams show that the duals of chordal nestohedra exhibit Gal's phenomenon and they give a combinatorial formula for the $\gamma_i$.

Let $\Delta_n$ be the simplex in $\rr^n$ with vertices $0,e_1,\ldots,e_n$, where $e_i$ is the $i^{th}$ standard basis vector.  The {\it stellohedron} $St_n$ is obtained from $\Delta_n$ by truncating all faces not containing $0$ in an order such that if $F,G$ are such faces and $\dim F<\dim G$ then $F$ is truncated before $G$.  Stellohedra are discussed in various papers, including \cite[Section 10.4]{prw} and \cite{CaDe}.

Stellohedra are simple polytopes.  Therefore, each dual polytope $St^\ast_n$ is a simplicial polytope.  If $F$ is a face of a polytope $P$ and $P_F$ is obtained from $P$ by truncating $F$, then $P_F^\ast$ is obtained from $P^\ast$ by stellar subdivision of the dual face $F^\ast$ (see for example \cite[Theorem 2.4]{Ew}).  Therefore, $St^\ast_n$ is (combinatorially equivalent to) the polytope obtained from $\Delta_n$ through stellar subdivision of all faces not contained in the convex hull of $\{e_1,\ldots,e_n\}$ in an order such that if $F,G$ are such faces and $\dim F<\dim G$ then $F$ is subdivided after $G$.

Postnikov, Reiner, and Williams \cite[Section 10.4]{prw} observe that 
$$h_{St^*_n}(t) = \tilde A_n(t).$$
Hence $\gamma$-positivity of $\tilde A_n(t)$ is a consequence of their general result on chordal nestohedra, as is their formula (\ref{PRWeq}).

Associated to any  simplicial polytope $P$ is a toric variety $X(P)$.   Danilov and Jurkiewicz (see \cite[eq.~(26)]{st2}) showed that for  any  simplicial polytope $P$,
$$h_P(t) = \sum_{j \ge 0} \dim H^{2j}(X(P)) t^j,$$ where   $H^{i}(X(P))$ is the degree $i$ singular cohomology of $X(P)$ over $\C$.  
From this, one has the  algebro-geometric interpretation of the Eulerian and binomial-Eulerian polynomials given by,
$$A_n(t) = \sum_{j = 0}^{n-1} \dim H^{2j}(X(P^*_n)) t^j$$
and
$$\tilde A_n(t) = \sum_{j = 0}^n \dim H^{2j}(X(St^*_n)) t^j.$$
The purpose of this section is to discuss equivariant versions of these interpretations.

Any simplicial action of a finite group $G$ on  $P$ determines an action of $G$ on $X(P)$ and thus a representation of $G$ on each cohomology group of $X(P)$.   If $G$ is the symmetric group $\sg_n$, the {\it Frobenius characteristic}, denoted by $\ch$ herein, assigns to each representation (up to isomorphism) of $G$ a symmetric function, as discussed in \cite[Section 7.18]{St}.  The symmetric group $\sg_n$ acts simplicially on $P_n^*$ and  $St_n^*$.
For $P = P_n^*$, Stanley \cite{st2}, using a recurrence of Procesi \cite{Pr}
obtained the  interpretation,
\begin{equation} \label{StanProEq} Q_n(\x,t) = \sum_{j=0}^{n-1} \ch(H^{2j} (X(P_n^*)) t^j .\end{equation}
From this interpretation, Stanley concluded that palindromicity and Schur-unimodality of $Q_n(\x,t)$ are  consequences of  an equivariant version  of the hard Lefschetz theorem.
Here, using (\ref{StanProEq}) and Procesi's technique, we obtain an analogous result for $\tilde Q_n(\x,t)$, which enables us to also interpret palindromicity and unimodality of $\tilde Q_n(\x,t)$  as a consequence of the equivariant version of the  hard Lefschetz theorem.

\begin{thm}  \label{StelloTh} For all $n \ge 1$,
$$\tilde Q_n(\x,t) = \sum_{j=0}^{n} \ch(H^{2j} (X(St_n^*)) t^j.$$
\end{thm}
\begin{proof} 
Let $\Delta_n$ be the $n$-simplex with vertex set $\{0, e_1,\dots, e_n\}$.  Let $\ff_i$ be the set of $i$-dimensional faces of   $\Delta_n$  containing $0$.  Let $T_n=\Delta_n$ and, for $1 \leq i \leq n-1$, let $T_i$ be the polytope obtained from $T_{i+1}$ by simultaneous stellar subdivision of all faces in $\ff_i$.  
Note that if $F \in {\mathcal F}_i$ then $F$ is a indeed face of $T_{i+1}$.  Moreover, the link $L_F$ of $F$ in the boundary complex of $T_{i+1}$ has one vertex for each face of the boundary of $\Delta_n$ strictly containing $F$.  Indeed, when applying stellar subdivision to such a face $E$, we remove $E$ and add a cone over the boundary of $E$.  Call the vertex of this cone $\phi(E)$.  The vertices of $L_F$ are all such $\phi(E)$, and a set $\{\phi(E_i)\}$ of such vertices forms a face of $L_F$ if and only if $\{E_i\}$ is a chain in the face poset of the boundary of  $\Delta_n$.  Thus $L_F$  is isomorphic to   the barycentric subdivision of the link of $F$ in the boundary of $\Delta_n$, which is equal to $\bar L_{F \setminus\{0\}}$, the barycentric subdivision of the link of $F\setminus\{0\}$ in the boundary of the $(n-1)$-simplex with vertex set $\{e_1,\dots,e_n\}$.  

Note that $T_1=St^\ast_n$. 
 The action of $\sg_n$ on $\{e_1,\ldots,e_n\}$ by permutation of indices induces a simplicial action on each $T_i$. 
  Thus we can consider the  representations of $\sg_n$ on the cohomology groups of the varieties $X(T_i)$.
  If $F = \{0,e_{i_1},\dots,e_{i_k} \}$, where $1\le i_1 < \dots < i_k \le n$, 
  then $\S_{[n] \setminus \{i_1,\dots, i_k\}}$ acts simplicially on
  $L_F$ and this action is equivalent to the action of $\S_{[n] \setminus\{i_1,\dots, i_k\}}$ on $\bar L_{F \setminus\{0\}}$.  By viewing $L_F$ and $\bar L_{F \setminus\{0\}}$ as simplicial polytopes, we have that these actions induce isomorphic representations of $\S_{[n] \setminus\{i_1,\dots, i_k\}}$ on cohomology of the corresponding varieties $X(L_F)$ and $X(\bar L_{F \setminus\{0\}})$.

For $1 \leq i \leq n$, we write $X_i$ for $X(T_i)$.  Then $X_n$ is the projective space $\pp^n$.  As explained in \cite[Section VI.7]{Ew}, $X_i$ is obtained from $X_{i+1}$ by a series of equivariant blowups.  For each $i \in \{1,\ldots,n\}$ and each $F \in \ff_i$, let $L_F$ be the link of $F$ in the boundary complex of $T_{i+1}$, as above.
As discussed in \cite[Section 3]{Pr}, there is an isomorphism of graded vector spaces,
\begin{equation} \label{isom}
H^\ast(X_i) \cong H^\ast(X_{i+1}) \oplus \bigoplus_{F \in \ff_i} H^\ast(X(L_F)) \otimes H^+(\pp^i),
\end{equation}
where $H^+({\pp^k}):=\oplus_{j>0}H^{2j}(\pp^{k})$.  

In fact, we can extend (\ref{isom}) to an isomorphism of $\sg_n$-representations.  Note that $\sg_n$ acts transitively on $\ff_i$, with the stabilizer of the face $F_i:={\mathsf c}{\mathsf o}{\mathsf n}{\mathsf v}\{0,e_1,\ldots,e_i\}$ being the subgroup $G_i:=\sg_{\{1,\ldots,i\}} \times \sg_{\{i+1,\dots,n\}}$.  
 The factor $\sg_{\{i+1,\ldots,n\}}$ in $G_i$ acts on $H^\ast(X(L_{F_i}))$ as it does on $H^\ast(X(\bar L_{F_i \setminus\{0\}}))$, as mentioned above.
 This is equivalent to the representation of $\sg_{n-i}$ on $H^*(X(P^*_{n-i}))$.   The factor $\sg_{\{1,\ldots,i\}}$ acts trivially on 
 $H^+(\pp^i)$, as explained in \cite[Section 3]{Pr}.

 We see now that the representation of $\sg_n$ on $H^\ast(X_i)$ is 
 the direct sum of the representation on $H^\ast(X_{i+1})$ with the representation induced from that of $G_i$ on $H^\ast(X(L_{F_i})) \otimes H^+(\pp^i)$ determined  by the representations of $\sg_{\{i+1,\ldots,n\}}$ and $\sg_{\{1,\ldots,i\}}$ 
 on the respective tensor factors.  Recalling the well known fact that $H^{2j}(\pp^i)$ has dimension one for $1 \leq j \leq i$ 
 and taking Frobenius characteristics, we obtain, for $1 \leq i \leq n$,
 
 $$
R_i(\x,t)=R_{i+1}(\x,t)+t[i]_th_i(\x)  \sum_{j=0}^{n-i-1} \ch(H^{2j} (X(P_{n-i}^*)) t^j ,
$$
where $$R_i(\x,t) := \sum_{j \ge 0} \ch(H^{2j}(X_i)) t^j.$$
By (\ref{StanProEq}) we may conclude that
 \begin{equation} \label{isom2}R_i(\x,t)=R_{i+1}(\x,t)+t[i]_th_i(\x)  Q_{n-i}(\x,t).
\end{equation}

By induction we have 
 $$
R_i(\x,t)=h_n(\x) [n+1]_t+\sum_{m=i}^{n-1}t[m]_th_{m}(\x) Q_{n-m}(\x,t) .
$$

Setting $i=1$ yields,

\begin{equation} \label{procmeth}
\sum_{j=0}^{n} \ch(H^{2j} (X(St_n^*)) t^j =h_n(\x) [n+1]_t+\sum_{m=1}^{n-1}t[n-m]_th_{n-m}(\x)Q_m(t). \end{equation}

We will  manipulate the symmetric function on the right side of (\ref{procmeth}) to obtain the desired result.  Setting $r=1$ in \cite[Corollary 4.1]{ShWa2}, we obtain

\begin{equation} \label{qeq}
Q_n(\x,t)=h_n(\x) +\sum_{k=0}^{n-2}Q_k(\x,t)h_{n-k}(\x) t[n-k-1]_t.
\end{equation}

Now
$$h_n(\x) [n+1]_t+\sum_{m=1}^{n-1}h_{n-m}(\x) Q_m(\x,t)t[n-m]_t \phantom{=h_n(\x) [n+1]_t+h_1(\x) Q_{n-1}(\x,t)t-h_nt[n]_t}$$
\vspace{-.2in}\begin{eqnarray*}
&=& h_n(\x) [n+1]_t+h_1(\x)Q_{n-1}(\x,t)t-h_n(\x)t[n]_t 
 \\& & +\sum_{m=0}^{n-2}h_{n-m}(\x)Q_m(\x,t)t[n-m]_t
  \\ & = & h_n(\x) +h_1(\x) Q_{n-1}(\x,t)t
   +\sum_{m=0}^{n-2}h_{n-m}(\x)Q_m(\x,t)t[n-m-1]_t 
   \\ & & +\sum_{m=0}^{n-2}h_{n-m}(\x)Q_m(\x,t)t^{n-m}
   \\ & = & Q_n(\x,t)+h_1(\x)Q_{n-1}(\x,t)t+\sum_{m=0}^{n-2}h_{n-m}(\x)Q_m(\x,t)t^{n-m} 
   \\ & = & \sum_{m=0}^{n}h_{n-m}(\x)Q_m(\x,t)t^{n-m},
\end{eqnarray*}
the third equality following from (\ref{qeq}).  
The result now follows from (\ref{procmeth}) and Proposition~\ref{newbinomprop}.
\end{proof}

\begin{cor}  For $0 \le j \le n-1$, $$\ps_q (\ch(H^{2j} (X(P_n^*)) ) =  \frac{a_{n,j}(q)} {(1-q) \dots (1-q^n)}$$
and for $0 \le j \le n$,
$$\ps_q (\ch(H^{2j} (X(St_n^*)) )=  \frac{\tilde a_{n,j}(q)} {(1-q) \dots (1-q^n)}.$$
\end{cor}

\begin{proof} The first equation is a consequence of (\ref{StanProEq}) and Theorem~\ref{SWps}, while the second equation is a consequence of Theorem~\ref{StelloTh} and Corollary~\ref{TildeSWps}.
\end{proof}

The next result follows from combining (\ref{StanProEq}) with Corollary~\ref{gescor} and combining Theorem~\ref{StelloTh} with Theorem~\ref{gesBcor}. 
\begin{cor} \label{equigalcor} For $P \in \{P_n^*, \, St_{n-1}^*\}$,  the polynomial
$\sum_{j=0}^{n-1} \ch(H^{2j} (X(P)) t^j$ is Schur-$\gamma$-positive.
\end{cor}

Corollary~\ref{equigalcor}  suggests an equivariant version of Gal's phenomenon.
\begin{defn} Let $P$ be a flag simplicial  $d$-dimensional polytope on which  a finite group $G$ acts simplicially.   The action of $G$ induces a graded representation of $G$ on cohomology of the associated toric variety $X(P)$.  We say that $(P,G)$ exhibits the {\em equivariant Gal phenomenon} if 
there exist $G$-modules $\Gamma_{P,k} $ such that  $$ \sum_{j=0}^d H^{2j}(X({P})) t^j = \sum_{k=0}^{\lfloor {d \over 2} \rfloor} \Gamma_{P,k} \,\,t^k (1+t)^{d-2k}.$$  
\end{defn}

Corollary~\ref{equigalcor} says that $(P_n^*,\sg_n)$ and $(St_n^*,\sg_n)$ both exhibit the equivariant Gal phenomenon.  

It is not the case that every group action on a flag simplicial polytope exhibits the equivariant Gal phenomenon.  Indeed, for $i \in [n]$, let $e_i$ be the $i^{th}$ standard basis vector in ${\mathbb R}^n$. Consider the cross-polytope $CP^n$, which is the convex hull of $\{\pm e_i:i \in [n]\}$.  It is straightforward to see that (the boundary of) $CP^n$ is a flag simplicial polytope.  The convex hull of some set $S$ of vertices of $CP^n$ is a boundary face if and only if there is no $i$ such that $S$ contains both $e_i$ and $-e_i$.

Let $T \leq GL_n({\mathbb R})$ be the group of all diagonal matrices whose nonzero entries are $1$ or $-1$ and let $S \leq GL_n({\mathbb R})$ be the set of all $n \times n$ permutation matrices.  The semidirect product $W=S \ltimes T$ preserves $CP^n$.  It is well known and not hard to see that the $h$-polynomial of $CP^n$ is $(1+q)^n$.  The action of $W$ on $H^0(X(CP^n))$ is trivial.  It follows that if $G \leq W$ and $(CP^n,G)$ exhibits the equivariant Gal phenomenon, then $G$ acts trivially on $H^\ast(X(CP^n))$.

Consider the element $c \in W$ satisfying $e_1c=e_{2}$, $e_2c=-e_1$ and $e_ic=-e_i$ for $3 \leq i \leq n$.  Note that $c$ and $c^3$ fix no boundary face of $CP^n$ and that $c^2$ fixes those boundary faces not including any of $\pm e_1$, $\pm e_2$.  It follows that the action of $C$ on $CP^n$ is proper, that is, the stabilizer in $C$ of any face $F$ of $CP^n$ fixes $F$ pointwise.  This allows us to apply results of Stembridge.  We observe that
$$
\det(I-qc)=(1+q^2)(1+q)^{n-2}.
$$
On the other hand, according to Theorem 1.4 and Corollary 1.6 of \cite{Ste2}, any $w \in W$ not having $1$ as an eigenvalue and acting trivially on $H^\ast(X(CP^n))$ satisfies
$$
\det(I-qw)=(1+q)^n.
$$
(Indeed, using the notation from \cite{Ste2}, any such $w$ satisfies $P_{\Delta^w}(q)=1$ and $\delta(w)=0$.)

We see that if $G \leq W$ contains (any conjugate of) $c$, then $(CP^n,G)$ does not exhibit the equivariant Gal phenomenon.  It would be interesting to find  classes, beyond $(P_n^*,\S_n)$ and $(St_n^*,\sg_n)$ that exhibit the equivariant Gal phenomenon.


\section{Remarks on  derangement polynomials} \label{DerSec}

One can modify the $q$-Eulerian polynomials $A_n(q,t)$ and $q$-Eulerian numbers $a_{n,j}(q)$ by summing over all derangements in $\sg_n$ instead of over all permutations in $\sg_n$.  That is, let $\mathcal D_n$ be the set of derangements in $\sg_n$ and 
 let
 $$D_n(q,t) := \sum_{\sigma \in \mathcal D_n}q^{\maj(\sigma) -\exc(\sigma)} t^{\exc(\sigma)},$$
for $n \ge 1$, and let $D_n(q,t) := 1$ for $n =0$.  Since $D_n(q,t) = A_n(q,t,0)$, it follows from (\ref{qfixEulereq}) that  
\begin{equation}\label{qDerEulereq} \sum_{n \ge 0 } D_n(q,t) \frac {z^n} {[n]_q!} = \frac{1-t} {\exp_q(tz) - t\exp_q(z)}.\end{equation}

Recall from the alternative proof of  Corollary~\ref{quniCor} that $D_n(q,t)$ is palindromic and $q$-unimodal.   
 (This result was first noted by the authors in \cite{ShWa2} and the $q=1$ case was proved earlier by  Brenti \cite{Br}.) There is  an analogous symmetric function result  conjectured by Stanley \cite{st2} and proved by Brenti \cite{Br}. The analogous  symmetric function result says that the symmetric function polynomial $Q^0_n(\x,y)$ is palindromic, Schur-positive and Schur-unimodal, where 
 $Q^0_n(\x,t)$ is defined by 
\begin{equation} \label{defDQ} \sum_{n \ge 0} Q^0_n(\x,t) z^n := \frac{1-t} {H(tz) - tH(z)}. \end{equation} 
An algebro-geometric interpretation of this result was given subsequently by Stanley (see \cite[page 825]{st4}), who determined the representation of the symmetric group on the graded local face module associated with the barycentric subdivision of the simplex.

A formula of Gessel shows that $Q^0_n(\x,t)$ is, in fact, Schur-$\gamma$ positive (see \cite[Equation (7.9)]{ShWa2}). 
  Let 
$$\gamma^{0}_{n,k}(\x) := \sum_{D \in  \mathcal H^{0}_{n,k}} s_D(\x),$$ where
$\mathcal H^0_{n,k}$ is  the set of skew hooks of size $n$ for  which $k$ columns have size 2 and  the remaining $n-2k$ columns, including the first and last column, have size 1.  Gessel's formula is equivalent  to 
\begin{equation} \label{DGesTh} Q^0_{n}(\x,t) = \sum_{k=0}^{\lfloor \frac {n-2} 2 \rfloor} \gamma^0_{n,k}(\x)  \, t^k (1+t)^{n-2-2k}, \end{equation}   for all $n \ge 1$.

 It is mentioned in \cite[Remark 5.5]{ShWa2}  that Gessel's formula
  can be used to establish $q$-$\gamma$-positivity of $D_n(q,t)$.  However an explicit description of the 
  $\gamma$-coefficients is not given there.  By applying
 stable principal specialization to (\ref{DGesTh}), one obtains the following description of the $\gamma$-coefficients.  This result is proved in \cite[Equation (1.3) and Theorem 3.3]{LiShWa} without the use of  Gessel's formula.  It appears also in \cite{LiZe}.
 
\begin{thm}  For  $0 \le k \le n $,   let $\Gamma^0_{n,k}$ be the set of permutations $\sigma \in \sg_n$ with no double descents, no intial descent, no final descent, and with $\des(\sigma)=k$. Let 
$$\gamma^0_{n,k}(q) := \sum_{\sigma \in \Gamma^0_{n,k}} q^{\inv(\sigma) } \,\,\,\,\,\, (= \sum_{\sigma \in \Gamma^0_{n,k}} q^{\maj(\sigma^{-1}) }).$$
Then \begin{equation} \label{qDFSeq} D_n(q,t) = \sum_{k=0}^{\lfloor \frac {n} 2 \rfloor} \gamma^0_{n,k}(q) \, t^k (1+t)^{n-2k}. \end{equation}  
Consequently,  $D_n(q,t)$ is $q$-$\gamma$-positive.
\end{thm}

As discussed in Stanley \cite{st4}, the Poincar\'e polynomial of the graded local face module associated with a certain type of subdivision of a simplicial complex is equal to the local $h$-polynomial associated with the subdivision, which in the case of  the barycentric subdivision of the $(n-1)$-simplex is  equal to $D_n(1,t)$.  In \cite{At1} Athanasiadis considers $\gamma$-positivity of local $h$-polynomials and formulates a generalization of Gal's conjecture for local $h$-polynomials, which would provide a geometric interpretation of $\gamma$-positivity of $D_n(1,t)$; see also \cite{At2,At3}.  
One could also consider an equivariant version of Gal's phenomenon in the local setting.

We remark that in \cite{LiShWa} the authors and Linusson consider  multiset versions of the Eulerian polynomial $A_n(t)$ and the derangement polynomial $D_n(1,t)$ and show that they  are $\gamma$-positive.  A  generalization of (\ref{eulereq}) is given   in \cite[Equation (5.4)]{LiShWa} and a generalization of the $q=1$ case of (\ref{qDFSeq})  is given in \cite[Equation (5.3)]{LiShWa}.

\section*{Acknowledgements}
We thank the referees for carefully reading the paper and providing useful comments.

\end{document}